\documentclass[11pt,espf]{article}

\usepackage{amsmath,amssymb,mathrsfs,theorem}
\usepackage{graphicx,subfigure,array,url}
\usepackage{epsfig,color,psfrag}
\usepackage[vlined,ruled,linesnumbered]{algorithm2e}
\SetKwInput{KwAssumes}{Assumes}
\newcommand{\algofont}[1]{{\scriptsize\textnormal{\textbf{#1}:}}}
\graphicspath{{./fig/}}

\newtheorem{theorem}{Theorem}[section]
\newtheorem{proposition}[theorem]{Proposition}

\newtheorem{definition}[theorem]{Definition}
\newtheorem{lemma}[theorem]{Lemma}

{\theorembodyfont{\rmfamily} 
\newtheorem{remark}[theorem]{Remark}

}

\def\QED{\mbox{\rule[0pt]{1.5ex}{1.5ex}}}

\newcommand{\Ti}{\mathbf{Ti}}

\newcommand{\QQ}{\mathcal{Q}}

\newcommand{\BB}{\mathcal{B}}

\newcommand{\VV}{\mathcal{V}}
\newcommand{\PP}{\mathcal{P}}

\newcommand{\X}{\mathcal{X}}
\newcommand{\s}{\mathcal{S}}

\newcommand{\env}{\mathcal{E}}
\newcommand{\p}{\mathbf{p}}
\newcommand{\q}{\mathbf{q}}
\newcommand{\w}{\mathbf{w}}
\newcommand{\z}{\mathbf{z}}
\newcommand{\real}{{\mathbb{R}}}
\renewcommand{\natural}{{\mathbb{N}}}
\newcommand{\subscr}[2]{{#1}_{\textup{#2}}}
\newcommand{\union}{\cup}
\newcommand{\intersection}{\cap}

\newcommand{\argmin}{\operatornamewithlimits{argmin}}

\newcommand{\sat}{\operatorname{sat}}
\newcommand{\map}[3]{#1: #2 \rightarrow #3}
\newcommand{\until}[1]{\{1,\dots,#1\}}

\newcommand{\norm}[1]{\|#1\|}
\newcommand{\abs}[1]{|#1|}

\newcommand\oprocendsymbol{\hbox{$\square$}}
\newcommand\oprocend{\relax\ifmmode\else\unskip\hfill\fi\oprocendsymbol}

\begin{document}
\title{On Vehicle Placement to Intercept Moving Targets\footnote{This
    material is based upon work supported in part by ARO MURI Award
    W911NF-05-1-0219 and ONR Award N00014-07-1-0721 and by the
    Institute for Collaborative Biotechnologies through the grant
    DAAD19-03-D-0004 from the U.S. Army Research Office. A preliminary
    version of this work titled ``Vehicle Placement to Intercept
    Moving Targets'' was presented at the 2010 American Control
    Conference, Baltimore, MD, USA.}}

\author{Shaunak D. Bopardikar\footnote{S. D. Bopardikar and F. Bullo are with
    the Center for Control, Dynamical Systems and Computation, University
    of California at Santa Barbara, Santa Barbara, CA 93106, USA; email:
    \{shaunak,bullo\}@engineering.ucsb.edu. S. L. Smith is with the
    Department of Electrical and Computer Engineering, University of
    Waterloo, Waterloo ON, N2L 3G1 Canada; email:
    stephen.smith@uwaterloo.ca }\and \qquad Stephen L. Smith \and \qquad %
  Francesco Bullo }

\date{}

\maketitle

\begin{abstract}
  We address optimal placement of vehicles with simple motion to
  intercept a mobile target that arrives stochastically on a line
  segment. The optimality of vehicle placement is measured through a
  cost function associated with intercepting the target. With a single
  vehicle, we assume that the target moves (i) with fixed speed and in
  a fixed direction perpendicular to the line segment, or (ii) to maximize
  the distance from the line segment, or (iii) to maximize intercept time. In each
  case, we show that the cost function is strictly convex, its
  gradient is smooth, and the optimal vehicle placement is obtained by
  a standard gradient-based optimization technique. With multiple
  vehicles, we assume that the target moves with fixed speed and in a
  fixed direction perpendicular to the line segment. We present a discrete
  time partitioning and gradient-based algorithm, and characterize
  conditions under which the algorithm asymptotically leads the
  vehicles to a set of critical configurations of the cost function.
\end{abstract}

\section{Introduction}
Vehicle placement to provide optimal coverage has received lot of
recent attention. This paper addresses vehicle placement scenarios
with the novelty of intercepting a mobile target generated randomly on
a segment. Applications of this work are envisioned in border patrol
wherein unmanned vehicles are placed to optimally intercept moving
targets that cross a region under surveillance
(cf.~\cite{ARG-ASH-JKH:04, RS-MK-KL-DC:08}).

Vehicle placement problems are analogous to geometric location
problems, wherein given a set of static points, the goal is to find
supply locations that minimize a cost function of the distance from
each point to its nearest supply location (cf.~\cite{EZ:85}). For a
single vehicle, the expected distance to a point that is randomly
generated via a probability density function, is given by the
continuous $1$--median function.  The $1$--median function is 
minimized by a point termed as the \emph{median}
(cf.~\cite{SPF-JSBM-KB:05}). For multiple distinct vehicle locations,
the expected distance between a randomly generated point and one of
the locations is known as the continuous multi-median function
(cf.~\cite{ZD:95}). For more than one location, the multi-median
function is non-convex, and thus determining locations that minimize
the multi-median function is hard in the general
case.~\cite{JC-SM-TK-FB:02j} addressed a distributed version of a
partition and gradient based procedure, known as the Lloyd algorithm,
for deploying multiple robots in a region to optimize a multi-median
cost function.~\cite{MS-DR-JJS:08} provided an adaptive control law to
enable robots to approximate the density function from sensor
measurements.~\cite{SM-FB:04p} presented motion coordination
algorithms to steer a mobile sensor network to an optimal
placement.~\cite{AK-SM:10} presented a coverage algorithm for vehicles
in a river environment. Related forms of the cost function have also
appeared in disciplines such as vector quantization, signal processing
and numerical integration (cf.~\cite{RMG-DLN:98, QD-VF-MG:99}).

In mobile target scenarios, the cost for the vehicle is a function of
relative locations, speeds and motion constraints considered. For an
adversarial target, the optimal vehicle motion is obtained by solving
a min-max pursuit-evasion game, in which the target seeks to maximize
while the vehicle seeks to minimize a certain cost function. The
vehicle strategy is a version of the classic proportional navigation
guidance law (cf.~\cite{MG:71}). With constraints such as a wall in
the playing space or non-zero capture distance, strategies with
optimal intercept time have been derived in~\cite{RI:65} and
in~\cite{MP:87}.

We consider a line segment on which a mobile target is generated via a
known spatial probability density and one or multiple vehicles seek to
intercept it. Knowledge about the density is a standard assumption in
search problems (cf.~\cite{LDS:75}). The goal is to determine vehicle
placements that minimize a cost function associated with the target
motion. With a single vehicle, we consider a class of cost functions
and establish its convexity, its smoothness and the existence of a
unique global minimizer. We show that the cost functions associated
with the target moving with fixed speed and in a fixed direction
perpendicular to the line segment, and with the target seeking to
maximize the distance from the segment, fall in the class of cost
functions that we have analyzed. The cost function for target motion
that maximizes the intercept time is shown to be proportional to the
continuous $1$--median function. With multiple vehicles and the target
moving with fixed speed perpendicular to the line segment, we first
provide an algorithm to partition the line segment among the vehicles
and characterize its properties. With the expected intercept time as
the cost, we propose a Lloyd algorithm in which every vehicle computes
its partition and descends the gradient of the cost computed over its
partition. We characterize conditions under which the vehicles
asymptotically reach a set of critical configurations.

In~\cite{SDB-SLS-FB:08v}, we addressed optimal placement for a single
vehicle with uniformly generated targets that have fixed speed and
direction. This paper extends our work to
include non-uniform generation density, adversarial target motion, and
multiple vehicle scenario. Existing analyses of Lloyd algorithms
(cf.~\cite{JC-SM-TK-FB:02j, QD-VF-MG:99}) do not apply to this
formulation due to a different form of the cost function.

This paper is organized as follows. The problem is formulated in
Section~\ref{secn:prob}. Single vehicle scenarios are addressed in
Section~\ref{secn:single}. The multiple vehicle scenario is addressed
in Section~\ref{secn:multi}. The proofs of
Lemmas~\ref{lem:cvx},~\ref{lem:grad_exptime} and~\ref{lem:unstable}
are presented in the Appendix.

\section{Problem Statement}\label{secn:prob}

We consider vehicles modeled with single integrator dynamics having
unit speed. A target is generated at a random position $(x,0)$ on the
segment $G:=[0,W]\times\{0\}$, termed the \emph{generator}, via a
specified probability density function $\map{\phi}{[0,W]}{\real_{\geq
    0}}$.~We assume that the density $\phi$ is bounded, i.e., there
exists an $M>0$ such that $\phi(x)\leq M, \forall x\in[0,W]$. The
target moves with bounded speed $v<1$, and is intercepted or captured
if a vehicle and the target are at the same point. We assume that the
vehicles can sense the instantaneous position and velocity vector of
the target. Target velocity information may be obtained using
Doppler-based methods. The goal is to determine vehicle placements and
corresponding capture motions that minimize a certain cost function
based on the maneuvering abilities of the target. We consider the
following cases.
\subsection{Single Vehicle Case}\label{secn:probsingle}
We determine a location $\p\in \real\times\real_{\geq 0}$ that
minimizes $\map{\subscr{C}{exp}}{\real \times \real_{\geq 0}}{\real}$ given by
\begin{equation}\label{eq:cost}
\subscr{C}{exp}(\p):= \int_0^W C(\p,x)\phi(x)dx, 
\end{equation}
where $\map{C}{\real\times\real_{\geq 0}\times[0,W]}{\real_{\geq 0}}$
is an appropriately defined cost of the vehicle position $\p$. In what
follows, we consider the following target motions.

(i) \emph{Constrained target:} We assume that the target
is \emph{constrained} to move in the positive $Y$-direction with fixed
speed $v<1$. From~\cite{SDB-SLS-FB:08v}, the cost function $C$ is
\begin{equation}\label{eq:time}
 T(\p,x) = \frac{\sqrt{(1-v^2)(X-x)^2 + Y^2}}{1-v^2} - 
    \frac{vY}{1-v^2},
\end{equation}
where the quantity $1$ arises from normalizing the vehicle speed to
unity, $\p := (X,Y)$ and $T$ is the time taken for the vehicle to
intercept the constrained target.

(ii) \emph{Adversarial target:} We consider a differential
pursuit evasion game in which the target (evader) seeks to maximize
and the vehicle (pursuer) seeks to minimize any one of the following
cost functions.

(a) \emph{Expected vertical height:} The cost function $C$ is the
vertical height $H(\p,x)$, i.e., the distance of the target from the
generator in the positive $Y$ direction, when the target is intercepted
(cf. Figure~\ref{fig:problem}).

(b) \emph{Expected intercept time}: The cost function $C$ is the time
interval $\Ti(\p,x)$ before the target is intercepted
(cf. Figure~\ref{fig:problem}). In this formulation, we also assume
that the target does not go below the X-axis.

The motions of the target and the vehicle are obtained from the
solution of these differential games and will be addressed, along with
formulae for $H$ and $\Ti$, in Section~\ref{secn:optadv}. 

\begin{figure}[!h]
\centering
\includegraphics[width=0.4\linewidth]{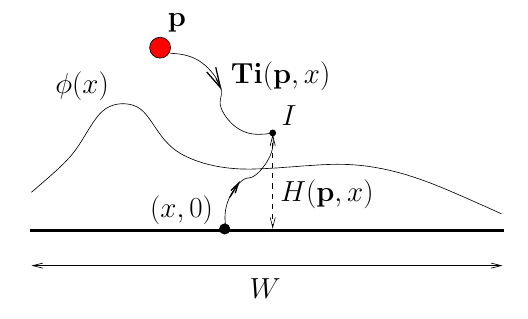}
\caption{Adversarial target scenario. The vertical height $H$ is the
  distance of the intercept point $I$ from the generator. Due to unit
  speed of the vehicle, the intercept time $\Ti$ is equal to the
  distance covered by $\p$ until the target is intercepted.}
\label{fig:problem}
\end{figure}

\subsection{Multiple Vehicles Case} \label{secn:probmulti}
We assume that the target translates in the positive $Y$-direction with
speed $v<1$. As shown in Figure~\ref{fig:multiprob}, given $m\geq 2$
vehicles having complete communication, the goal is to determine vehicle
locations $\p_i \in [0,W]\times\real_{\geq 0}$, for every $i\in
\{1,\dots,m\}$, that minimize the expected constrained travel time given by
\begin{equation}\label{eq:multtime}
\subscr{T}{exp}(\p_1,\dots,\p_m) :=
\int_0^W\min_{i\in\{1,\dots,m\}}T(\p_i,x)\phi(x)dx,
\end{equation}
where $T(\p_i,x)$ is given by Eq.~\eqref{eq:time}. 

\begin{figure}[!h]
\centering
\includegraphics[width=0.4\linewidth]{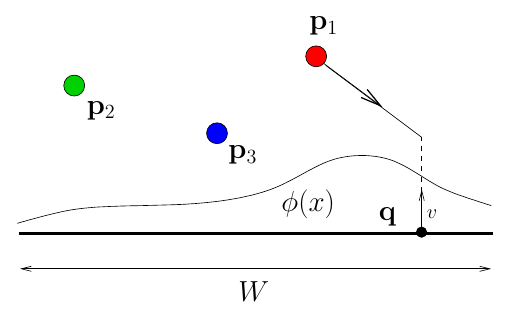}
\caption{Intercepting a target moving with fixed speed
  $v$ perpendicular to the generator.}
\label{fig:multiprob}
\end{figure}

\begin{remark}[Adversarial target scenarios] The adversarial target
  counterparts of the multiple vehicle case are difficult to analyze
  theoretically. Section~\ref{secn:adv} provides some insight into the scenarios.
\end{remark}

\section{Single Vehicle Scenarios} \label{secn:single}
In this section, we address optimal placement for a single vehicle in
the scenarios mentioned in Section~\ref{secn:probsingle}.

\subsection{Cost functions for Expected Constrained Travel Time and
  Expected Vertical Height}\label{secn:general}
We analyze cost functions given by Eq.~\eqref{eq:cost}, where the
function $C$ has the form
\begin{equation}\label{eq:C}
C(X,Y,x):= a\sqrt{b(X-x)^2 + Y^2} - cY,
\end{equation}
and $a$, $b$, and $c$ are positive constants, with $a>c$. The
constrained travel time $T$ (cf. Eq.~\eqref{eq:time}) has this form and we will show that the
vertical height $H$ also has this form in Section~\ref{secn:optadv}.

Using $C$ as defined in Eq.~\eqref{eq:C}, the partial derivatives
of $\subscr{C}{exp}(X,Y)$ from Eq.~\eqref{eq:cost} are given by
\begin{equation}\label{eq:partx}
\frac{\partial \subscr{C}{exp}}{\partial X} =
ab\int_0^W\frac{(X-x)\phi(x)}{\sqrt{b(X-x)^2 + Y^2}}dx,
\end{equation}
\begin{equation}\label{eq:party}
\frac{\partial \subscr{C}{exp}}{\partial Y} =
aY\int_0^W\frac{\phi(x)}{\sqrt{b(X-x)^2 + Y^2}}dx-c.
\end{equation}

We first establish strict convexity of $\subscr{C}{exp}$.

\begin{lemma}[Strict convexity of expected cost]\label{lem:cvx}
In the domain ${]0,W[}\times\real_{>0}$, the expected cost
$\subscr{C}{exp}(X,Y)$ 
\begin{enumerate}
\item is continuous and convex in
  $X$ and $Y$, and
\item has a unique minimizer.
\end{enumerate}
\end{lemma}

Statement (i) of Lemma~\ref{lem:cvx} involves showing that the Hessian
matrix of $C$ with respect to $X$ and $Y$ is positive
semi-definite. The existence of a minimizer is established by showing
that the partial derivatives of $C$ with respect to $X$ and $Y$ vanish
inside ${]0,W[}\times\real_{>0}$.  The uniqueness is established by
assuming two distinct minimizers, using the convexity of $C$ and the
necessary conditions for the two distinct locations to be minima to
reach a contradiction.

Lemma~\ref{lem:cvx} leads to the main result.

\begin{theorem}[Minimizing expected cost]\label{thm:opt}
  From any initial vehicle location in $\real\times \real_{>0}$ and for any
  pobability density function $\phi$, assume that the vehicle motion obeys
  \begin{equation}\label{eq:graddes}
    \dot{\p}(t) = - \frac{\partial \subscr{C}{exp}}{\partial \p}(\p(t)).
  \end{equation}
  The following statements hold:
  \begin{enumerate}
  \item the vehicle position $\p(t)$ remains in $[0,W]\times\real_{>0}$ at
    all times, and
  \item the vehicle position $\p(t)$ converges to the unique global
    minimizer of $\subscr{C}{exp}$.
  \end{enumerate}
\end{theorem}

Theorem~\ref{thm:opt} answers the problem of minimizing the expected
value of $T$, given by Eq.~\eqref{eq:time}, with $a:= 1/(1-v^2)$,
$b:=(1-v^2)$ and $c:=v/(1-v^2)$, and $a>c$. Except for special cases such as in Remark~\ref{rem:spl}, it is difficult to provide analytical expressions for the
minimizer $\p^*$ of $\subscr{T}{exp}(\p^*)$. 

\begin{remark}[Equal target and vehicle speeds]\label{rem:spl}
If the target and the vehicle have equal speeds, the optimal placement
in $X$ which minimizes $\subscr{T}{exp}$ is at the \emph{centroid} of the distribution $\phi$, with the optimal $Y$ given by
\begin{align*}
X^* = {\int_0^W\phi(x)xdx}; \, Y^* = \sqrt{{\int_0^W\phi(x)(X^*-x)^2dx}}. \quad
\oprocend
\end{align*}
\end{remark}

\subsection{Optimal Placement for Adversarial Target}\label{secn:optadv}
We now address the differential pursuit-evasion games
stated in case (ii) of Section~\ref{secn:probsingle}, and 
determine optimal vehicle placements.

\subsubsection{Minimizing the Expected Vertical Height}
Given an ordered pair of distinct points $\{\q_1,\q_2\}$ in a plane and a
scalar $\lambda \in {[0,1[}$, \emph{Apollonius circle}
(cf.~\cite{RI:65}) is the set of points $\w$ in the plane that satisfy
$\norm{\q_2-\w}= \lambda\norm{\q_1-\w}$.  Letting the pursuer position
$\p=:\q_1$, the target (evader) position be $\q_2$ and $v=:\lambda$, the
following is an established result.

 \begin{proposition}[Apollonius circle during pursuit]\label{prop:appo}
If the pursuer and the evader both travel straight toward a point $U$
on the Apollonius circle, then any new such circle, obtained from a
pair of simultaneous intermediate positions of the pursuer and the
evader, is tangent to the original circle at U, and is
contained in the original circle.
\end{proposition}  

For any evader strategy, the pursuit strategy that minimizes the
vertical height (cf.~\cite{RI:65}) is to choose the pursuer's
velocity vector such that the line joining the pursuer and the evader
remains parallel at all times to the line joining their initial
locations, while reducing the distance. So, for optimal vehicle placement, it
suffices to determine the optimal evader
strategy. Algorithm~\ref{algo:mtt} summarizes the optimal evader strategy,
shown in Figure~\ref{fig:mtt}.

\begin{algorithm}[h] 
  \KwAssumes{Pursuer at $(X,Y)$. Evader at $(x,0)$.} %

Compute center and radius of the Apollonius circle:
\begin{align*}
O &:= (O_x,O_y)=\Big(\frac{x-v^2X}{1-v^2},\frac{-v^2Y}{1-v^2}\Big), \\ 
R &:= \frac{v}{1-v^2}\sqrt{(X-x)^2+Y^2}. 
\end{align*}

Move towards the point $(O_x,O_y+R)$ with speed $v$.

  \caption{\bf Move towards top-most}
\label{algo:mtt}
\end{algorithm}

\begin{figure}[!h]
\centering
\includegraphics[width=0.5\linewidth]{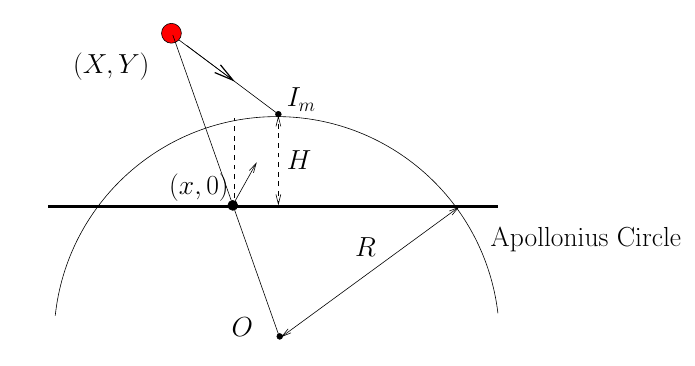}
\caption{The evader strategy that maximizes $H$ is to move towards the
  top-most point on the Apollonius circle drawn initially.}
\label{fig:mtt}
\end{figure}

The following result is immediate from Proposition~\ref{prop:appo}.
\begin{lemma}[Optimality of Move towards top-most]\label{lem:mtt}
The strategy \emph{move towards top-most} is the evader's optimal
strategy and the optimal vertical height is
\[
H(X,Y,x) = \frac{v}{1-v^2}\sqrt{(X-x)^2+Y^2}-\frac{v^2Y}{1-v^2}.
\]
\end{lemma}
Comparing the expression for $H$ given by Lemma~\ref{lem:mtt} with the
definition of $C$ in Eq.~\eqref{eq:C}, we have $a:=
v/(1-v^2)$, $b:=1$ and $c:=v^2/(1-v^2)$, and $a>c$ since
$v<1$. Thus, by applying Theorem~\ref{thm:opt}, we obtain
the following result.

\begin{theorem}[Minimizing expected height]
  From an initial location in $\real\times \real_{>0}$, assume that
  the vehicle motion obeys Eq.~\eqref{eq:graddes} with
  $\subscr{C}{exp}$ replaced by $\subscr{H}{exp}$, then the vehicle
  position $\p(t)$ converges to the unique global minimizer of
  $\subscr{H}{exp}$

\end{theorem}

\subsubsection{Minimizing the Expected Intercept Time}
The underlying differential game in this formulation is the classic
\emph{wall pursuit} game (cf.~\cite{RI:65}). We
present the main result for completeness.

\begin{lemma}[Wall Pursuit game]\label{lem:wall}
The evader strategy that maximizes the intercept time is to move
towards the furthest point of the Apollonius circle on the $X$-axis.
\end{lemma}

\begin{figure}[!h]
\centering
\includegraphics[width=0.5\linewidth]{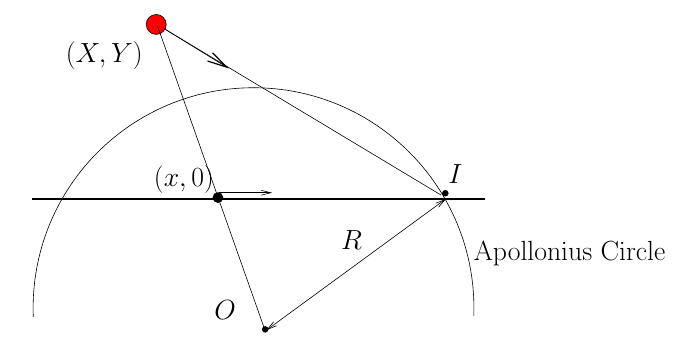}
\caption{The evader strategy that maximizes $\Ti$ is to move towards
  the furthest point on the Apollonius circle drawn initially. }
\label{fig:intercept}
\end{figure}
This optimal evader strategy is illustrated in Figure
\ref{fig:intercept}. Now, given a convex region $\QQ\subset \real$ and a density function
$\map{\psi}{\QQ}{\real_{\geq 0}}$, the \emph{median}
(cf.~\cite{SPF-JSBM-KB:05}) is the unique global
minimizer $p^*$ of $\int_{\QQ} \abs{p-z}\psi(z)dz$.

We now present the main result of this section.

\begin{theorem}[Optimality of the Median]
The median of the region $[0,W]\times \{0\}$ with the density
function $\phi$ uniquely minimizes the expected intercept time.
\end{theorem}
\begin{proof} 
From Lemma~\ref{lem:wall} and Pythagoras theorem, 
\[
\Ti(X,Y,x) = \frac{1}{v}\sqrt{R^2-\Big(\frac{vY}{1-v}\Big)^2} + \frac{1}{v}
\left|\frac{x-vX}{1-v}-x\right|,
\]
where $R$ is the radius of the Apollonius circle drawn at the initial
instant. Since pursuer placement on the $X$-axis results into
decreasing the intercept time $\Ti$, we have
\[
\subscr{\Ti}{exp}(X,0) = \frac{v+3}{1-v^2}\int_0^W\abs{X-x}\phi(x)dx,
\]
which is minimized uniquely by the median.
\end{proof}

\section{The Case of Multiple Vehicles} \label{secn:multi}
We now address the multi-vehicle case from Section~\ref{secn:probmulti}.

\subsection{Dominance Region Partition}
We introduce a generator partitioning procedure by defining dominance
regions between each pair of vehicles\footnote{Here, a partition of
  $[0,W]$ is a collection $\{\VV_1,\dots,\VV_m\}$ of subsets of
  $[0,W]$ such that $\VV_1\cup\ldots \cup \VV_m = [0,W]$ and
  $\VV_i\cap \VV_j$ has zero length, for all distinct $i,j$.}. We will see that the resulting
partition allows us to write the cost in Eq.~\eqref{eq:multtime} in a
simplified form as in Eq.~\eqref{eq:texp}.

\begin{definition}[Pairwise dominance region]\label{def:dominance} 
For $i,j \in \{1,\dots,m\}$, the pairwise dominance
  region $U_{ij}\subseteq [0,W]$ of $\p_i$ with respect to $\p_j$ is the
set of initial target locations for which $\p_i$ takes \emph{lesser}
time to intercept the target than $\p_j$:
\[
U_{ij}:=\{ x\in[0,W]  \, | \,T(\p_i,x)\leq  T(\p_j,x)\}.
\]
\end{definition}

We now describe a procedure (summarized in Algorithm~\ref{algo:pair})
to determine $U_{ij}$. Without loss of generality, assume that
$X_i<X_j$. If $Y_i=Y_j$, i.e., the vehicles are at the same distance
from the generator, then $U_{ij}$ is the piece of $G$ that lies in the
half-plane that is formed by the perpendicular bisector of the segment
joining $\p_i$ and $\p_j$ and which contains $\p_i$.  Now if
$Y_i<Y_j$, then we look for points $(x,0)$ in $G$ for which
$T(\p_i,x)\leq T(\p_j,x)$. By setting $(1-v^2)=:b$,
Eq.~\eqref{eq:time} gives
\begin{equation}\label{eq:pair}
\sqrt{b(X_i-x)^2+Y_i^2}-vY_i \leq \sqrt{b(X_j-x)^2+Y_j^2}-vY_j.
\end{equation}
Eq.~\eqref{eq:pair} simplifies to a quadratic in $x$ having real
roots, which provides at most two points for the boundary between
$U_{ij}$ and $U_{ji}$. To determine these boundary points, consider the
perpendicular bisector of the segment joining $\p_i$ and $\p_j$, as
shown in Figure~\ref{fig:geom}.
\begin{figure}[!h]
\centering 
\includegraphics[width=0.5\linewidth]{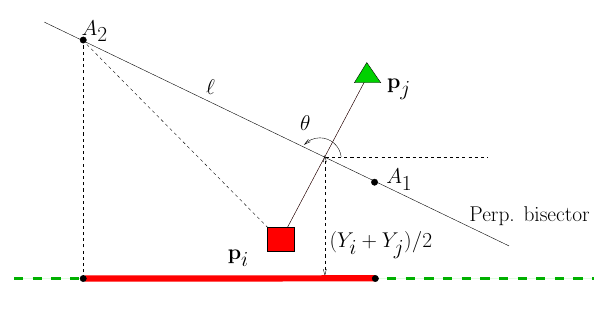}
\caption{Pairwise dominance regions. The
  solid and dashed dominance regions belong to the square (red) and
  the triangle (green) vehicles respectively.}
\label{fig:geom}
\end{figure}
We look for points $A_1$ and $A_2$ on this bisector such that the
distances of these points from the real line is $v$ times their
respective distances from the vehicles. This leads to the
following quadratic in $\ell$
 \[
4(\sin^2{\theta}-v^2)\ell^2 + 4(Y_i+Y_j)\sin{\theta}\ell = - (Y_i+Y_j)^2+v^2\norm{\p_i-\p_j}^2, 
\]
where $\ell$ and $\theta:= \arctan_2((Y_j-Y_i),(X_j-X_i)) + \pi/2$ are
as shown in Figure~\ref{fig:geom}. Let $\ell_1$ and $\ell_2$ be the
roots of the above quadratic. Then the $Y$-coordinates of the
candidate boundary points $A_1$ and $A_2$ are given by
\[
[y_1 ,\, y_2]^T = [1 ,\,1]^T(Y_i+Y_j)/2 + [\ell_1 ,\,\ell_2]^T\sin{\theta}.
\]
Now, $A_1$ and $A_2$ are both boundary points if and only if both have
positive $Y$-coordinates. It can be shown that there exists at least
one among them which has positive $Y$-coordinate. There arise two cases: 

(i) If there are two candidate points $A_1$ and $A_2$ (as in
Figure~\ref{fig:geom}), then we look at their corresponding $X$
coordinates, $(x_1,x_2)$ given by Step 9. For $(x,0)\in
G\cap[x_1,x_2]\times\{0\}$, we have $T(\p_i,x)\leq T(\p_j,x)$, and
thus $U_{ij}$ is
$G\cap[x_1,x_2]\times\{0\}$.

(ii) If there is only one candidate point $A_1$, then we look at its
$X$ coordinate, $x_1$ given by Step 13. By assumption $X_i<X_j$, and
hence for $(x,0)\in G\cap[-\infty,x_1]\times\{0\}$, we have
$T(\p_i,x)\leq T(\p_j,x)$ and thus $U_{ij}$ is $ G\cap[-\infty,x_1]\times\{0\}$.

\begin{algorithm}[h] 
  \KwAssumes{Distinct $\p_i=(X_i,Y_i)$, $\p_j=(X_j,Y_j)$.} %
  \eIf{$Y_i=Y_j$,}{
    $U_{ij} := \begin{cases}
      [0,(X_i+X_j)/2], \, \, \, \text{ if } X_i<X_j \\
      [(X_i+X_j)/2,W], \text{ if } X_i>X_j
    \end{cases}$
  }
  {
    $\theta:= \arctan_2(Y_j-Y_i, X_j-X_i) + \pi/2$\\[.3em]
    $\ell_{1,2} :=$ two roots of $0=4(\sin^2(\theta)-v^2)\ell^2$ \\ \quad$+
    4(Y_i+Y_j)\sin(\theta)\ell + (Y_i+Y_j)^2-v^2\norm{\p_i-\p_j}^2$\\[.3em]
    $y_{1,2} := (Y_i+Y_j)/2 + \sin(\theta)\ell_{1,2}$ \\[.3em]

    \eIf{$y_1>0$ \textup{and} $y_2>0$}
    {$x_{1,2} := (X_i+X_j)/2 + \cos(\theta)\ell_{1,2}$\\
       $ U_{ij}:= \begin{cases} [0,W]\cap[x_1,x_2],  \qquad \, \, \text{ if } Y_i<Y_j \\
          [0,W] \setminus {]x_1,x_2[}, \qquad \, \, \, \text{ if } Y_i\geq Y_j. 
        \end{cases}$
    }
    {$k :=$ index in $\{1,2\}$ for which $y_k>0$\\
      $x := (X_i+X_j)/2 + \ell_k\cos(\theta)$

      $U_{ij}:= \begin{cases} [0,W]\cap{]{-\infty},x]}, \qquad \text{ if }
        X_i<X_j \\
          [0,W]\cap{[x,+\infty[}, \qquad \text{ if } X_i\geq X_j
        \end{cases}$
    }   
    }
  \caption{\bf Pairwise Dominance Region}
  \label{algo:pair}
\end{algorithm}

Thus, we have established the following property.
\begin{proposition}[Pairwise dominance region]\label{prop:pair}
Given distinct locations $\p_i=(X_i,Y_i)$, $\p_j=(X_j,Y_j)$, if a target arrives at $(x,0)$, where $x\in U_{ij}$ generated by Algorithm~\ref{algo:pair}, then $T(\p_i,x)\leq T(\p_j,x)$.
\end{proposition}

Similar to pairwise dominance regions, we introduce the concept of
\emph{dominance region} $\VV_i\in \PP([0,W])$ for the $i$th vehicle,
for every $i\in\{1,\dots,m\}$, which is the set of $X$-coordinates of
target locations for which $\p_i$ takes the \emph{minimum} time to
intercept of all vehicles.

Assuming complete communication between vehicles,
Algorithm~\ref{algo:pair} is extended to determine the dominance region
for a vehicle by (i) determining pairwise dominance regions between
vehicles and, (ii) taking intersection of all pairwise dominance
regions, as presented in Algorithm~\ref{algo:dominance}.
 
\begin{algorithm}[h] 
  \KwAssumes{Distinct locations $\{\p_1,\dots,\p_m\}$.} %
  \ForEach{\textup{vehicle} $j\in\{1,\dots,m\}\setminus\{i\}$,}{
  Determine $U_{ij}$ using Algorithm~\ref{algo:pair}. \\
  }
  $\VV_i = \bigcap_{j=1,\dots,m, j\neq i}U_{ij}$. \; %
\caption{\bf Dominance region}
\label{algo:dominance}
\end{algorithm}

Algorithm~\ref{algo:dominance} with three vehicles is illustrated in
Figure~\ref{fig:dominance}. The next result follows due to disjoint interiors of dominance regions, and due to Proposition~\ref{prop:pair}.

\begin{figure}[!h]
\centering
\includegraphics[width=0.7\linewidth]{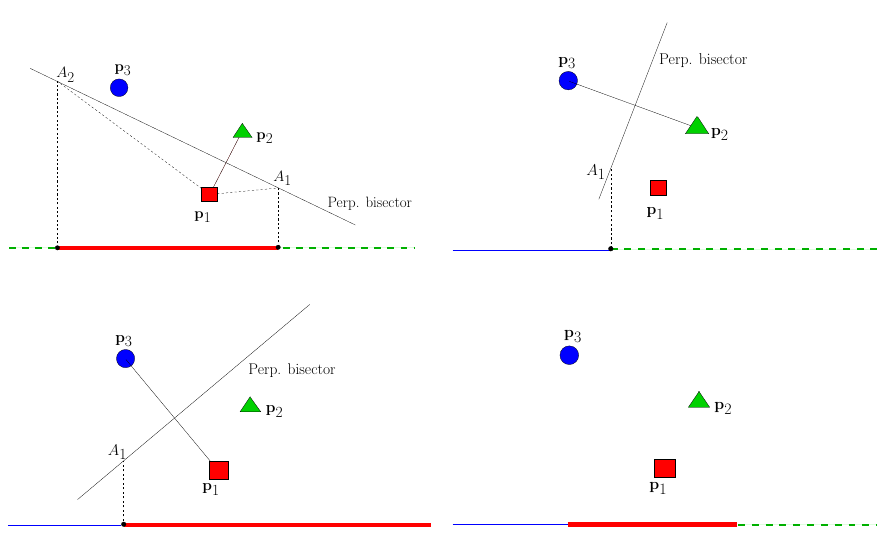}
\caption{Algorithm~\ref{algo:dominance} with three vehicles. The
  thick solid, the
  dashed and the thin solid dominance regions belong to the square (red), the
  triangle (green) and the round (blue) vehicles, respectively.}
\label{fig:dominance}
\end{figure}

\begin{proposition}[Optimality of dominance regions]\label{prop:dom}
Given distinct vehicle positions and a target arrival, 

(i) the dominance regions generated by
  Algorithm~\ref{algo:dominance} form a partition of the generator.

(ii) The time taken to reach the target is minimized by the vehicle whose dominance region contains the target arrival location. 
\end{proposition}

The dominance region of a vehicle may be empty if the vehicle is very
far from the generating line as compared to the other vehicles
(cf. first part of Figure~\ref{fig:nonunif}). Conversely, all
vehicles have non-empty dominance regions when they are all in
$[0,W]\times\real_{\geq 0}$ and have the same $Y$-coordinates.

Since the dominance regions are set-valued functions of vehicle
positions, we now provide some background on continuity of set-valued
functions, and establish continuity of the dominance regions. This
property will be used to analyze our gradient-based procedure. Let
$\env := [0,W]\times\real_{\geq 0}$, let $\PP([0,W])$ be the set of all
subsets of $[0,W]$, let $\BB(r)$ be the closed ball of radius $r$ around
the origin, and let $+$ denote the Minkowski sum of two sets. The domain of
a set-valued map $F:X\rightrightarrows Z$ is the set of all $\q\in X$
such that $F(\q)\neq \emptyset$. $F$ is \emph{upper} (resp. \emph{lower})
\emph{semi-continuous} in its domain if, for every $\q$ in its domain and for
every $\epsilon>0$, there exists a $\delta>0$ such that for every
$\z\in \q +\BB(\delta)$, $F(\z)\subset F(\q)+B(\epsilon)$
(resp. $F(\q)\subset F(\z)+B(\epsilon)$). $F$ is \emph{continuous} if it is
both upper and lower semi-continuous.

The pairwise dominance region between $\p_i$ and $\p_j$ is a set
valued function $U_{ij}:\env^2\setminus\s_{ij} \rightrightarrows
\PP([0,W])$, where $\s_{ij}\subset \env^2$ is the set of coincident
locations for $\p_i$ and $\p_j$. Similarly, the dominance region
for vehicle $i$ is a set-valued map
$\VV_{i}:\env^m\setminus\s_{i}\rightrightarrows \PP([0,W])$,
where $\s_{i}\subset\env^m$ is the set of vehicle locations in which
at least one other vehicle is coincident with $\p_i$. We now
show that the dominance regions vary continuously with the vehicle
positions.

\begin{proposition}[Continuity of dominance regions]\label{prop:usc}

(i) For every distinct $i$ and $j$ in the set $\{1,\dots,m\}$, the set
  valued map $U_{ij}$ is continuous in $\env^2\setminus\s_{ij}$.

(ii) For each vehicle $i \in \{1,\dots,m\}$, the set valued map $\VV_{i}$
is continuous on its domain.
\end{proposition}
\begin{proof}
The roots of Eq.~\eqref{eq:pair} which is a quadratic in $x$, vary
continuously with $\p_i$ and $\p_j$. Thus, the map $U_{i,j}$ is
continuous in $\env^2\setminus\s_{ij}$.

The domain of $\VV_i$ is contained in the domain of $U_{ij}$ for every
$j\neq i$. From part (i) of this Proposition, for every $j\neq i$,
the set-valued map $U_{ij}$ is upper semi-continuous in
$\env^2$. Thus, for every $j\neq i$, at every $\q$ in the domain of
$\VV_i$ and for every $\epsilon>0$, there exist $\delta_{ij}>0$ such
that for every $\z\in \q +\BB(\delta_{ij})$, $U_{ij}(\z)\subset
U_{ij}(\q)+B(\epsilon)$. Given an $\epsilon>0$, by the choice of
$\delta_i=\min\{\delta_{ij}, \forall j\neq i\}$, we obtain that for
every $\z\in \q +\BB(\delta_{i})$, $\VV_{i}(\z)\subset
\VV_{i}(\q)+B(\epsilon)$. Thus $\VV_i$ is upper semi-continuous. Lower
semi-continuity of $\VV_i$ is established similarly and the result follows.
\end{proof}

\subsection{Minimizing the Expected Constrained Travel Time} \label{secn:exptime}
For distinct vehicle locations, Eq.~\eqref{eq:multtime} can be written
as
\begin{equation}\label{eq:texp}
\subscr{T}{exp}(\p_1,\dots,\p_m) = \sum_{i=1}^m\int_{\VV_i}T(\p_i,x)\phi(x)dx,
\end{equation}
where $\VV_i$ is the dominance region of the $i$th vehicle. The
gradient of $\subscr{T}{exp}$ is computed using the following formula,
which allows each vehicle to compute the gradient of
$\subscr{T}{exp}$ by integrating the gradient of $T$ over
$\VV_i$. 
\begin{lemma}[Gradient computation]\label{lem:grad_exptime}
  For all vehicle configurations such that no two vehicles are at
  coincident locations, the gradient of the expected time with respect to
  vehicle location $\p_i$ is
  \[
  \frac{\partial \subscr{T}{exp}}{\partial \p_i} = \int_{\VV_i}\frac{\partial T}{\partial \p_i}(\p_i,x)\phi(x)dx. 
  \]
\end{lemma}
Akin to similar results in~\cite{FB-JC-SM:09}, the proof involves
writing the gradient of $\subscr{T}{exp}$ as a sum of two contributing
terms. The first is the final expression, while the second is a number
of terms which cancel out due to continuity of $T$ at the boundaries
of dominance regions.

For $\z \in \real^2$, let
$\map{\sat}{\real^2}{\real^2}$ denote the saturation function, i.e.,
if $\norm{\z}\leq 1$, then $\sat(\z) = \z$; otherwise, $\sat(\z) =
\z/\norm{\z}$. Inspired by the established Lloyd algorithm
(cf.~\cite{FB-JC-SM:09}), we present a discrete-time descent approach
in Algorithm~\ref{algo:lloyd}. The idea is to minimize
$\subscr{T}{exp}$ by making each vehicle follow gradient descent over
its dominance region. If the dominance region is empty, then the
vehicle moves towards the generator, until it obtains a non-empty
dominance region.

\begin{algorithm}[h] 
  \KwAssumes{Distinct locations $\{\p_1,\dots,\p_m\}\in\env^m$} %

  \ForEach{\textup{time} $t\in\natural$}%
  { Compute $\VV_i(t)$ by Algorithm~\ref{algo:dominance} as a function of
    $\{\p_1(t),\dots,\p_m(t)\}$\\
    \eIf{$\VV_i(t)$ \textup{is empty},}{%
      Move in unit time to $(X_i,Y_i-\min\{1,Y_i\})$\; %
    }{ %
      For $\tau\in[t,t+1]$, move according to
      $\displaystyle{\dot{\p}_i(\tau) = -\sat\Big(\int_{\VV_i(t)}\frac{\partial}{\partial
        \p_i}T(\p_i(\tau),x)\phi(x)dx\Big)}$ %
}
}
\caption{\bf Lloyd descent for vehicle $i$}
\label{algo:lloyd}
\end{algorithm}

Next, we define critical configurations for the vehicles, which means
that every vehicle is at the unique minimizer of the cost evaluated
over its dominance region.
\begin{definition}[Critical configuration]\label{def:cdrc}
  A set of locations $\{\p_1,\dots,\p_m\}$ is a \emph{critical
    configuration} if,
  \begin{equation*}
  \p_i = \argmin_{\z\in\env}\int_{\VV_i} T(\z,x)\phi(x)dx,    
  \end{equation*}
  for all $i \in \until{m}$, where $\{\VV_1,\dots,\VV_m\}$ is the dominance region partition induced
  by $\{\p_1,\dots,\p_m\}$.
\end{definition}

We now state the main result of this section, that gives
sufficient conditions under which the vehicles asymptotically reach
a critical configuration using Algorithm~\ref{algo:lloyd}.
\begin{theorem}[Convergence of Lloyd descent]\label{thm:lloyd}
  Let $\map{\gamma}{\natural}{\real^{2m}}$ be the evolution of the $m$ vehicles
  according to Algorithm~\ref{algo:lloyd} and assume that no two vehicle
  locations become coincident in finite time or asymptotically. The
  following statements hold:

(i) the expected travel time $t\mapsto\subscr{T}{exp}(\gamma(t))$ is a 
  non-increasing function of time;

(ii) if the dominance region $\VV_i$ of any vehicle $i$ is empty at some
  time, then $\VV_i$ will be non-empty within a finite time; and

(iii) if there exists a time $t$ such that every dominance region is
  non-empty for all times subsequent to $t$, then the vehicle
  locations converge to the set of critical dominance region
  configurations.
\end{theorem}

The assumptions of non-coincidence of vehicle locations and the
non-emptiness of the dominance regions after a finite time ensure that
the dominance regions are continuous functions of vehicle
positions. The continuity of dominance regions in turn allows the
LaSalle Invariance principle to be applicable. Further, it may be
possible that the dominance region of a vehicle keeps alternating
between being empty and non-empty under the action of
Algorithm~\ref{algo:lloyd}. However, it was observed through numerous
simulations (e.g., see Figure~\ref{fig:nonunif}) that after a finite
time, the dominance region of every vehicle remained non-empty.

\begin{proof}[Proof of Theorem~\ref{thm:lloyd}]
 We begin by showing statement (i). In every iteration of
 Algorithm~\ref{algo:lloyd}, step~\algofont{2} does not increase the
 expected time $\subscr{T}{exp}$ due to the optimality of the
 dominance region partition, by
 Proposition~\ref{prop:dom}. Step~\algofont{4} does not change the
 $\subscr{T}{exp}$ as the associated dominance region is
 empty. Finally, step~\algofont{6} does not increase $\subscr{T}{exp}$
 as the vehicle is moving along the gradient descent flow of
 $\subscr{T}{exp}$. Thus, the expected time is non-increasing under
 Algorithm~\ref{algo:lloyd}.

 Statement (ii) follows from the fact that whenever $\VV_i=\emptyset$ for
 vehicle $i$, due to step~\algofont{4}, vehicle $i$ reaches the
 generator after finite time and therefore has a non-empty $\VV_i$.

 For non-empty $\VV_i$, let
 $\map{\mathcal{A}}{\X\times\PP([0,W])}{\X}$, be the flow map of the
 differential equation at step~\algofont{6} from time $t$ to time
 $t+1$. For statement (iii), consider the discrete-time dynamical system
 given by the tuple $(\X,\X_0,\mathcal{A})$, where $\X = \env^m$ and
 $\X_0\in \env^m$ is the set of initial vehicle positions.

We now apply the discrete-time LaSalle Invariance Principle
  (Theorem 1.19 in \cite{FB-JC-SM:09}), for which we verify
  the four assumptions as follows.

  1. Existence of a positively invariant set: At every iteration of
  step~\algofont{6}, each vehicle follows saturated gradient descent
  of a cost function belonging to the class of Eq.~\eqref{eq:C} over
  its dominance region fixed for the iteration. By the first statement of
  Theorem~\ref{thm:opt}, each vehicle remains in $\env$ throughout
  the iteration, and therefore at all times. Thus, the set $\env^m$ is
  positively invariant for the system $(\X,\X_0,\mathcal{A})$.

  2. Existence of a non-increasing function along $\mathcal{A}$:
  $\subscr{T}{exp}$ is non-increasing along $\mathcal{A}$, by
  statement (i) of this theorem.

  3. Boundedness of all evolutions of $(\X,\X_0,\mathcal{A})$:
  Gradient descent keeps the $X$ coordinates bounded in $[0,W]$. It
  remains to show that the $Y$-coordinates of all vehicles remain
  bounded. Let us suppose the contrary. Then, there are two cases: (a)
  there exists a sequence of times on which at least one vehicle has
  its location bounded and at least one other vehicle, say vehicle
  $k$, has its $Y$-coordinate growing without limits; or (b) there
  exists a sequence of times on which the $Y$-coordinates of all
  vehicles grow unbounded. In case (a), after finite time, the
  dominance region $\VV_k$ becomes empty, thus contradicting the
  assumption of statement (iii) of this theorem. If case (b) occurs,
  then there exists a subsequence of $\subscr{T}{exp}$ which grows
  unbounded, thus contradicting statement (i) of this theorem. Thus,
  all evolutions of $(\X,\X_0,\mathcal{A})$ are bounded.

  4. Continuity of $\subscr{T}{exp}$ and $\mathcal{A}$: Continuity
  of $\subscr{T}{exp}$ follows from Eq.s~\eqref{eq:time} and
  \eqref{eq:texp}. To verify continuity of $\mathcal{A}$, note that
  whenever $\VV_i$ is non-empty, by Proposition~\ref{prop:usc},
  $\VV_i$ is continuous with respect to vehicle locations.  Thus, as
  long as $\VV_i$ is non-empty, $\mathcal{A}$ is continuous as the
  integrand is continuous with respect to vehicle locations.
  
  By LaSalle Invariance Principle, the evolutions of 
  $(\X,\X_0,\mathcal{A})$ converge to a set of the form
  $\subscr{T}{exp}^{-1}(\kappa)\intersection \mathcal{M}$, where
  $\kappa$ is a real constant and $\mathcal{M}$ is the largest
  positively invariant set in $\{x\in \X \,|\,
  \subscr{T}{exp}(\mathcal{A}(x))=\subscr{T}{exp}(x)\}$. Since
  $\subscr{T}{exp}$ remains constant under action of $\mathcal{A}$ for
  the set of critical configurations, it is contained in a set of the
  form $\subscr{T}{exp}^{-1}(\kappa)\intersection \mathcal{M}$. If a
  set of vehicle positions is not critical, then $\subscr{T}{exp}$
  strictly decreases under the action $\mathcal{A}$, and therefore the
  set of vehicle positions is not contained in a set of
  $\subscr{T}{exp}^{-1}(\kappa)\intersection \mathcal{M}$. Thus, the
  vehicles converge to the set of critical configurations.
\end{proof}   

The next result gives a simple condition to identify an unstable critical
 configuration, which is an unstable equilibrium of
Algorithm~\ref{algo:lloyd}. Figure~\ref{fig:critical} illustrates this
result.

\begin{lemma}[Disconnected partitions are unstable]\label{lem:unstable}
A critical configuration is unstable if some vehicle has a
disconnected dominance region.
\end{lemma}

The proof involves perturbing the position of a vehicle with a
disconnected dominance region, and then showing that the 
gradient in the $X$ direction for that vehicle takes the vehicle away
from the equilibrium configuration.

\subsection{Simulations}\label{secn:simu}

We now present some simulations of Algorithm~\ref{algo:lloyd}.

\noindent \emph{Examples of critical locations:} We consider two
vehicles, and a uniform target generation density, i.e.,
$\phi(x) = 1/W$. From initial locations such as in the leftmost of
Figure~\ref{fig:critical} wherein both vehicles having the same
$X$-coordinate of $W/2$, but different $Y$-coordinates, the vehicles
asymptotically approach the configuration in the center
figure. However, a small perturbation to the vehicles leads to the
configuration in the rightmost figure. Thus, this simulation
illustrates Lemma~\ref{lem:unstable}. However, from most initial conditions,
the vehicles converged to a critical configuration as in the rightmost
figure.

\begin{figure}[h]
\centering 
\includegraphics[width=0.7\linewidth]{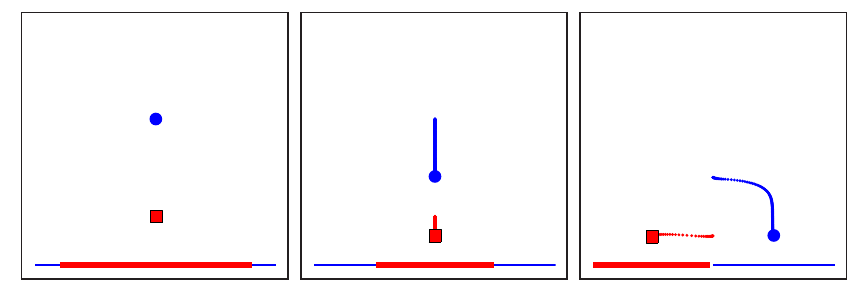}
\caption{Algorithm~\ref{algo:lloyd} for uniform arrival density. The thick
  solid (red) and the thin solid (blue) dominance regions belong to the
  square and the round vehicles respectively. 
  The
  vehicles first tend to an critical configuration (center figure). A
  perturbation to their positions makes them move to a stable critical
  configuration.} %
\label{fig:critical} %
\end{figure}%

\noindent \emph{Non-uniform probability distribution:}
We consider three vehicles and the arrival probability
density function,
\begin{align*}
\phi(x) = \begin{cases} {8x}/{W^2}, &\text{ if } x\in{[0,W/4]}, \\
{2}/{W} - {8}(x-{W}/{4})/(3W^2), &\text{ if } x\in{]W/4,W]}.
\end{cases}
\end{align*}   

Initially, the round vehicle had an empty dominance region
(Figure~\ref{fig:nonunif}, left). After finite time, the round
vehicle obtained a non-empty dominance region
(Figure~\ref{fig:nonunif}, center), after which all vehicles continued
to have non-empty dominance regions. Thus, by Theorem~\ref{thm:lloyd},
the vehicles converged to a critical configuration
(Figure~\ref{fig:nonunif}, right).

\begin{figure}[h]
\centering 
\includegraphics[width=0.7\linewidth]{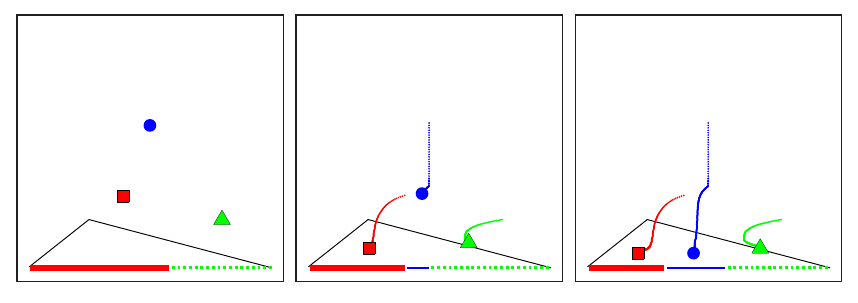}
\caption{Algorithm~\ref{algo:lloyd} for non-uniform arrival density
  (black line). The dashed dominance region belongs to the triangle
  vehicle. Initially, the round vehicle has no dominance region.} %
\label{fig:nonunif} %
\end{figure}%

\subsection{Adversarial target motion: An insight} \label{secn:adv} We
now show that the partitioning procedure can be applied to the case of
adversarial target. However, the dominance regions are difficult to
characterize analytically. To see this, consider
only two vehicles $\p_1$ and $\p_2$, and suppose the target is generated at
$\q$. In order to maximize the vertical height (resp. intercept
time), the optimal strategy for the target is (cf.~\cite{RI:65}):

(i) Compute the Apollonius circles (ACs) with respect to the pursuer
  locations $\p_1$ and $\p_2$.

(ii) Move to the point with highest $Y$ coordinate
  (resp. furthest point) in the intersection of the ACs.

  The strategy is illustrated in Figure~\ref{fig:adv}. We partition
  the generator into: (i) $\VV_1$ (resp. $\VV_2$): Set of all locations
  $\q$ for which the AC of the target with respect to pursuer $\p_1$
  (resp. $\p_2$) is entirely contained in the AC with respect to
  pursuer $\p_2$ (resp. $\p_1$); and (ii) $\VV_{1,2}:= G \setminus
  (\VV_1\union \VV_2)$.

\begin{figure}[!h]
\centering
\includegraphics[width=0.8\linewidth]{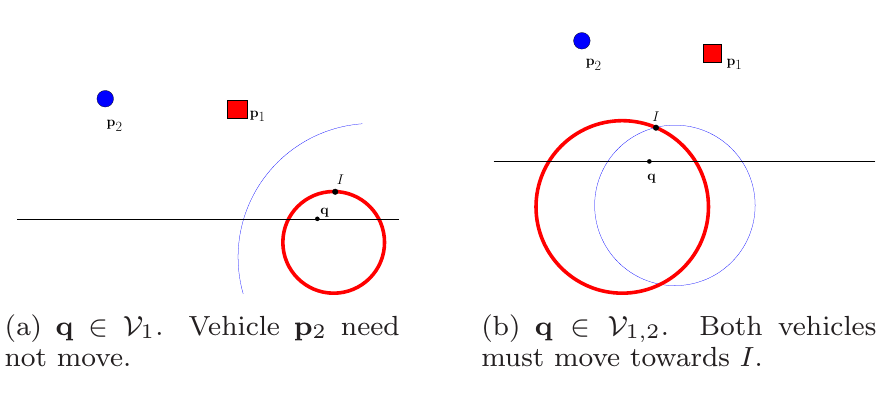}
\caption{Generator partitioning to intercept an adversarial target.}
\label{fig:adv}
\end{figure}
The optimal pursuit strategy is to move $\p_i$ if $\q \in \VV_i$, for
some $i\in\{1,2\}$, and to move both $\p_1$ and $\p_2$ if $\q \in
\VV_{1,2}$. The average cost $\subscr{C}{exp}(\p_1,\p_2)$ can be
written as
\begin{equation*}
\subscr{C}{exp} = \sum_{i=1}^2\int_{\VV_i}C(\p_i,x)\phi(x)dx + \int_{\VV_{1,2}}C(\p_1,\p_2,x)\phi(x)dx.
\end{equation*}

The gradient of the first term is similar to that in
Lemma~\ref{lem:grad_exptime}. But, the second term
$C(\p_1,\p_2,x)$ is difficult to characterize in the form of an
analytical expression. Thus, the adversarial case is a
challenging future direction.

\section{Conclusions and Future Directions}
 We addressed the problem of optimally placing vehicles having simple
 motion in order to intercept a mobile target that arrives
 stochastically on a line segment. For a single vehicle, we determined
 unique optimal placements when target motion was either constrained,
 i.e., with fixed speed and direction, or adversarial. For the
 multiple vehicle scenario and with constrained motion targets, we
 characterized conditions under which a partition and gradient based
 algorithm takes the vehicles asymptotically to the set of critical
 points of the cost function.

 A future direction is to consider numerical methods for the
 multiple vehicles with adversarial targets. Another direction is to
 consider stochastic target motion.


\section*{Appendix}

In this Appendix, we provide complete proofs of
Lemmas~\ref{lem:cvx},~\ref{lem:grad_exptime} and~\ref{lem:unstable}.

\medskip 

\noindent\hspace{0em}{\itshape Proof of Lemma~\ref{lem:cvx}:}
  The first claim follows by verifying that the Hessian matrix of $C$
  with respect to $X$ and $Y$ is positive semi-definite in the domain
  ${]0,W[}\times\real_{>0}$.

For the second claim, we need to show existence and uniqueness of a
minimizer in the domain ${]0,W[}\times\real_{>0}$.

1. Existence: We show that a minimizer cannot lie on the boundary or
outside of the domain $[0,W]\times\real_{\geq 0}$. We begin by showing
that $Y^*$ exists and is finite. Taking the limit of
$\subscr{C}{exp}(X,Y)$ as $Y\to +\infty$,
\[
\liminf_{Y\to+\infty} \subscr{C}{exp}(X,Y) \geq \liminf_{Y\to+\infty}
(a-c)Y\int_0^W \phi(x)dx = +\infty,
\]
since by assumption, $a>c$. Thus, $Y^*$ exists and is finite.

Finally, to show that a minimizer lies in ${]0,W[}\times\real_{> 0}$, we
need to prove two statements: (a) $Y^*\neq 0$, and (b) $X^*\in {]0,W[}$. To
show (a), Eq.~\eqref{eq:party} along with the assumption $\phi(x)\leq M$,
for every $x\in [0,W]$, yields
\begin{align*}
\frac{\partial \subscr{C}{exp}}{\partial Y} &\leq
MYa\int_0^W\frac{dx}{\sqrt{b(X-x)^2 + Y^2}}
-c \\
&\leq
\frac{MYa}{\sqrt{b}}(\log(W+\sqrt{W^2+Y^2/b}) - \log(Y/\sqrt{b})) -c.
\end{align*}
Thus, $\limsup_{Y\to 0^+}{\partial \subscr{C}{exp}}/{\partial Y} \leq
-c$. Thus, for $Y$ near zero, the gradient of
$\subscr{C}{exp}$ points in the negative $Y$-direction, implying that
$Y^*\neq 0$. 

To show (b), we first observe that for a given $Y$, in the limit as $X
\to \pm\infty$, $\subscr{C}{exp} \to +\infty$, and therefore $X^*$
must be bounded. Finally, the claim follows since the partial
derivative of $\subscr{C}{exp}$ with respect to $X$ is strictly
negative for $X\leq 0$ and is strictly positive for $X\geq W$.

Facts (a) and (b) coupled with convexity of $\subscr{C}{exp}$ with
respect to $X$ and $Y$ establish the existence part.

2. Uniqueness: Let there be two locations $(X_1,Y_1)$ and $(X_2,Y_2)$ that minimize
the expected cost. Since the
expected cost $\subscr{C}{exp}$ is convex in $X$ and $Y$, a convex
combination of $(X_1,Y_1)$ and $(X_2,Y_2)$ also minimizes
$\subscr{C}{exp}$. Thus, the necessary conditions for minimum are
satisfied by \\
$(\bar{X}(\alpha),\bar{Y}(\alpha)):=(\alpha X_1+(1-\alpha)X_2,\alpha
Y_1+(1-\alpha)Y_2)$, for every $\alpha \in [0,1]$. Thus,
\begin{align*}
\int_0^W\frac{(\bar{X}(\alpha)-x)\phi(x)}{\sqrt{(b\bar{X}(\alpha)-x)^2+\bar{Y}(\alpha)^{2}}}dx &= 0,\\
\int_0^W\frac{\bar{Y}(\alpha)\phi(x)}{\sqrt{b(\bar{X}(\alpha)-x)^2+\bar{Y}(\alpha)^{2}}}dx &= \frac{c}{a}.
\end{align*}
Since the above conditions hold for every $\alpha \in
[0,1]$, the partial derivatives of the above conditions evaluated at
$\alpha = 0$, must equal zero, which yields
\begin{align*}
\int_0^W \frac{(X_2-x)Y_2(Y_1-Y_2) - Y_2^2(X_1-X_2)}{(b(X_2-x)^2+Y_2^{2})^{3/2}}\phi(x)dx = 0,\\
\int_0^W \frac{(X_2-x)Y_2(X_1-X_2) -
  (Y_1-Y_2)(X_2-x)^2}{(b(X_2-x)^2+Y_2^{2})^{3/2}}\phi(x)dx \\ = 0,
\end{align*}
where $\phi(x)/(b(X_2-x)^2+Y_2^{2})^{3/2}=:f(X_2,Y_2,x)$ is strictly positive for $Y_2>0$. Multiplying the first equation by
$(X_1-X_2)$, the second by $(Y_1-Y_2)$, and adding the 
equations, 
\[
\int_0^W f(X_2,Y_2,x)(Y_2(X_1-X_2) - (X_2-x)(Y_1-Y_2))^2dx = 0.
\]
Since $f(X_2,Y_2,x) \geq 0$, we must have $Y_2(X_1-X_2) -
(X_2-x)(Y_1-Y_2) = 0$, \emph{for every} $x$ at which
$f(X_2,Y_2,x)>0$, which is feasible only if $X_1-X_2 = 0$ and $Y_1-Y_2
= 0$.

Parts 1 and 2 complete the proof for the second claim.
\hspace*{\fill}~\QED\par\endtrivlist\unskip

\medskip

\noindent\hspace{0em}{\itshape Proof of Lemma~\ref{lem:grad_exptime}:} 
 Let $\p_j$ be termed as a \emph{neighbor} of $\p_i$, i.e., $j\in$
  neigh$(i)$, if $\VV_i\cap\VV_j$ is non-empty. Then,
\[
\frac{\partial \subscr{T}{exp}}{\partial \p_i} = \frac{\partial}{\partial \p_i}\int_{\VV_i}T(\p_i,x)\phi(x)dx + \sum_{j\text{ neigh }(i)}\frac{\partial}{\partial \p_i}\int_{\VV_j}T(\p_j,x)\phi(x) dx,
\]
Now, let $\VV_i = \bigcup_{l=1,\dots,n_i}[a_{l},b_{l}]$, for some
finite integer $n_i$. Then, there are two cases:

1. \emph{Every boundary point in the interior of $[0,W]$ belong to the
  dominance region of exactly two vehicles:} In this case, all
boundary points $a_l$ and $b_l$ are differentiable with respect to
$\p_i$. Therefore, by Leibnitz's Rule\footnote{Leibnitz's
  Rule: \begin{equation*}\frac{\partial}{\partial z}\int_{a(z)}^{b(z)}
    f(z,x)dx = \int_{a(z)}^{b(z)} \frac{\partial f(z,x)}{\partial z}dx
    + f(z,b)\frac{\partial b(z)}{\partial z} - f(z,a)\frac{\partial
      a(z)}{\partial z}.\end{equation*}},
\[
\frac{\partial}{\partial \p_i}\int_{\VV_i}T(\p_i,x)\phi(x)dx = \int_{\VV_i}\frac{\partial T}{\partial \p_i}(\p_i,x)\phi(x)dx 
+ \sum_{l=1}^{n_i} T(\p_i,b_{l})\frac{\partial b_{l}}{\partial \p_i} - T(\p_i,a_{l})\frac{\partial a_{l}}{\partial \p_i}.
\]
Unless $a_1 = 0$, or $b_{n_i}=W$ (in which case the partial
derivatives with respect to $\p_i$ are zero), for every
$l\in\{1,\dots,n_i\}$, there exist some $j\in$ neigh$(i)$ and some
$k\in$ neigh$(i)$, such that
\begin{align}\label{eq:db}
\frac{\partial}{\partial \p_i}\int_{\VV_j}T(\p_j,x)\phi(x) dx &= -T(\p_j,b_l)\frac{\partial b_{l}}{\partial \p_i}, \\
\text{and, }\frac{\partial}{\partial \p_i}\int_{\VV_k}T(\p_k,x)\phi(x) dx &=
T(\p_k,a_l)\frac{\partial a_{l}}{\partial \p_i}, \nonumber
\end{align}
where we have made use of Leibnitz's Rule. Due to the continuity of $T$ at the boundary points, we obtain
\[
T(\p_j,b_l) = T(\p_i,b_l), \quad T(\p_k,a_l) = T(\p_i,a_l),
\]
and on summation,
\[
 \sum_{j \in \text{ neigh }(i)}\frac{\partial}{\partial \p_i}\int_{\VV_j}T(\p_j,x)\phi(x) dx
+ \sum_{l=1}^{n_i} T(\p_i,b_{l})\frac{\partial b_{l}}{\partial \p_i} - T(\p_i,a_{l})\frac{\partial a_{l}}{\partial \p_i} = 0.
\]
The proof is complete for this case.

2. \emph{Some $b_l$ (or $a_l$) belongs to the dominance regions of
  $\p_i$ and at least two other vehicles:} Let $\p_j$ and $\p_r$ be
two of these vehicles. We perturb the position of $\p_i$ in a
direction $v$ by a small distance $\epsilon$. We claim that the
boundary term $T(\p_i,b_l)\partial b_l/\partial \p_i$ is cancelled
independent of the choice of the direction $v$. The following two
possibilities arise (cf.~Figure~\ref{fig:gradient}): either the point
$b_l$ moves to the right or $b_l$ moves to the left by some distance
$\delta(\epsilon)$. The steps for the former possibility are exactly identical
to Case 1.  In the latter possibility, we can write
Eq.~\eqref{eq:db}, if the interval $[b_l-\delta,b_l]$ belongs to
$\VV_j(\p_1,\dots,\p_i+\epsilon v,\dots,\p_m)$, or we can write
Eq.~\eqref{eq:db} with $j$ replaced by $r$, if the interval
$[b_l-\delta,b_l]$ belongs to $\VV_r(\p_1,\dots,\p_i+\epsilon
v,\dots,\p_m)$. Thus, in both of these possibilities, the steps from case 1
apply leading to the cancellation of all the boundary terms.

This completes the proof.
\hspace*{\fill}~\QED\par\endtrivlist\unskip

\begin{figure}[!h]
\centering
\includegraphics[width=0.7\linewidth]{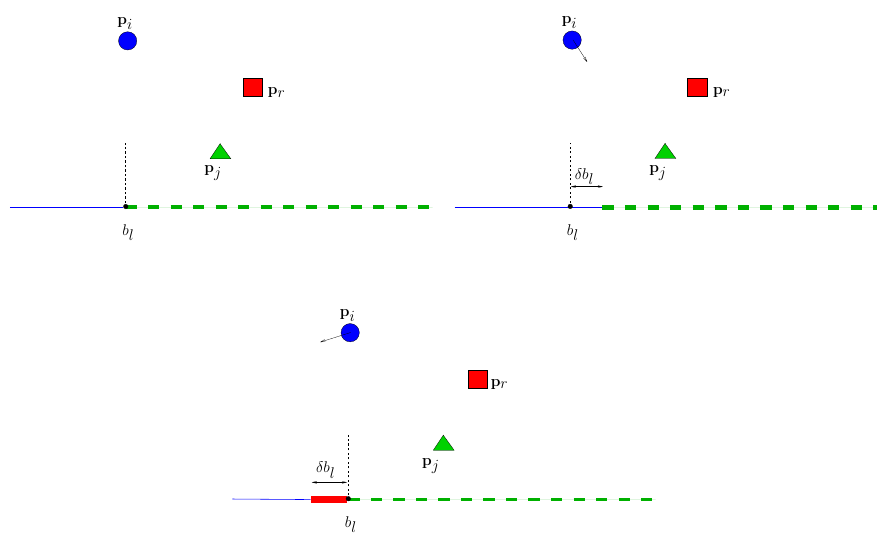}
\caption{Illustration of case 2 in the proof of
  Lemma~\ref{lem:grad_exptime}. In the first figure, the point $b_l$ belongs to the
  dominance regions of $\p_i$, $\p_j$ and $\p_r$. The second
  and third figures show how the dominance regions would change
  when $\p_i$ is perturbed in different directions. The
  thick solid, the
  dashed and the thin solid dominance regions belong to the square (red), the
  triangle (green) and the round (blue) vehicles, respectively.}
\label{fig:gradient}
\end{figure}

\medskip

\noindent\hspace{0em}{\itshape Proof of Lemma~\ref{lem:unstable}:} 
We first prove the result for the case of two
  vehicles. Let $\{\p_1^*,\p_2^*\}$ be a critical dominance region
  configuration in which $\p_1$ has its dominance region
  disconnected. Let the dominance regions be 
\begin{align*}
\VV_1 = [A,B], \quad \VV_2 = [0,A]\union[B,W].
\end{align*}
Perturb $\p_1$ by a small distance $\delta X$in the positive $X$ direction,
as shown in Figure~\ref{fig:unstable}.
\begin{figure}[!h]
\centering 
\includegraphics[width=0.5\linewidth]{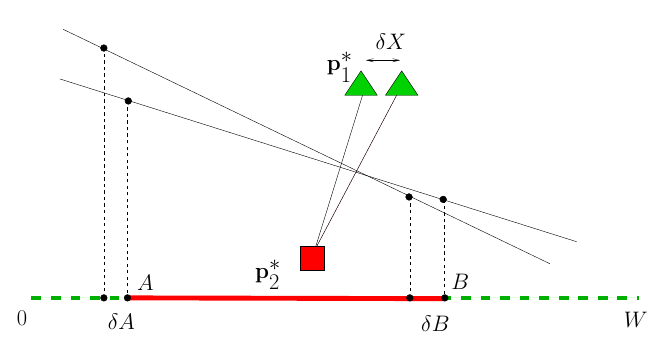}
\caption{Illustrating instability of critical dominance region
  configuration having a disconnected dominance region.} 
\label{fig:unstable}
\end{figure}
From continuity of the dominance regions, we have
\begin{align*}
\VV_1 = [A-\delta A,B-\delta B], \,
\VV_2 = [0,A-\delta A]\union[B-\delta B,W],
\end{align*}
where $\delta A$ and $\delta B$ are positive and sufficiently
small. Let $I_2(x,X_2, Y_2):=(X_2-x)\phi(x)/\sqrt{(1-v^2)(X_2-x)^2 +
    Y_2^2}$. Evaluating the partial derivatives in $X$ for $\p_2$,
\[
(1-v^2)\frac{\partial}{\partial X}T_2 = \int_{A-\delta A}^{B-\delta B}I_2dx,
\]
Upon further simplification,
\begin{align}\label{eq:unstable}
(1-v^2)\frac{\partial T_2}{\partial X} &= \int_{A- \delta A}^{A}I_2dx+ \int_{A}^{B}I_2dx - \int_{B-\delta B}^bI_2dx \nonumber \\
&= \int_{A-\delta A}^AI_2dx - \int_{B-\delta B}^{B}I_2dx,
\end{align}
as the sum of the first two terms is zero from the necessary condition
for $\p_2^*$ to be a minimum. For sufficiently small $\delta X$, there
exist $\delta A$ and $\delta B$ such that $I_2>0$ on the interval
$[0,A-\delta A]$ and $I_2<0$ on the interval $[B-\delta B,B]$, since
$X_2-x>0$ and $X_2-x<0$ respectively on the two intervals. Thus,
$\partial T_2/\partial X > 0$ when $\p_1$ is displaced to $(X_1+\delta
X,Y_1)$, which implies the direction of gradient descent in $X$ is the
negative $X$ direction for $\p_2$, and similarly is the positive $X$
direction for $\p_1$. Thus, the critical configuration is unstable.

In the case of $m>2$ vehicles, let the dominance regions of $\p_1$ and
$\p_2$ share at least one point. Since both terms on the right hand
side of Eq.~\eqref{eq:unstable} are positive, the direction of
gradient descent in $X$ is the negative direction for $\p_2$
independent of the fact whether the dominance regions of $\p_1$ and
$\p_2$ share one or two common points.
\hspace*{\fill}~\QED\par\endtrivlist\unskip

\end{document}